\documentclass[10pt]{amsart}
\usepackage{amssymb,amstext,amsmath,amscd,amsthm,amsfonts,enumerate,latexsym,stmaryrd,multicol,geometry,graphicx,mathrsfs}
\usepackage[usenames]{color}
\usepackage[all]{xy}
\newtheorem{thm}{Theorem}[section]
\newtheorem{lem}[thm]{Lemma}
\newtheorem{prop}[thm]{Proposition}
\newtheorem{cor}[thm]{Corollary}
\theoremstyle{definition}
\newtheorem{dfn}[thm]{Definition}

\newtheorem{rem}[thm]{Remark}
\newtheorem{conv}[thm]{Convention}

\newtheorem{ex}[thm]{Example}

\theoremstyle{remark}

\newtheorem*{claim*}{Claim}

\numberwithin{equation}{thm}
\def\A{\mathrm{A}}
\def\add{\operatorname{\mathsf{add}}}
\def\C{\mathcal{C}}
\def\CC{\mathbb{C}}
\def\ch{\operatorname{char}}
\def\cm{\mathsf{CM}}
\def\cone{\operatorname{cone}}
\def\D{\mathrm{D}}
\def\d{\operatorname{\mathscr{D}}}
\def\db{\operatorname{\mathsf{D^b}}}
\def\depth{\operatorname{depth}}
\def\ds{\operatorname{\mathsf{D_{sg}}}}
\def\E{\mathrm{E}}
\def\ge{\geqslant}
\def\Hom{\operatorname{Hom}}
\def\height{\operatorname{ht}}
\def\ind{\operatorname{\mathsf{ind}}}
\def\ker{\operatorname{Ker}}
\def\lcm{\operatorname{\underline{\mathsf{CM}}}}
\def\le{\leqslant}
\def\m{\mathfrak{m}}
\def\Min{\operatorname{Min}}
\def\mod{\operatorname{\mathsf{mod}}}
\def\n{\mathfrak{n}}
\def\nf{\operatorname{NF}}
\def\P{\mathbb{P}}
\def\p{\mathfrak{p}}
\def\q{\mathfrak{q}}
\def\r{\mathfrak{r}}
\def\reg{\operatorname{Reg}}

\def\sing{\operatorname{Sing}}
\def\spec{\operatorname{Spec}}
\def\syz{\Omega}
\def\T{\mathcal{T}}

\def\thick{\operatorname{\mathsf{thick}}}
\def\U{\mathcal{U}}
\def\V{\mathrm{V}}
\def\X{\mathcal{X}}
\def\Y{\mathcal{Y}}
\def\Z{\mathbb{Z}}
\begin{document}
\title{On a Verdier quotient of a derived category of a local ring}
\author{Yuki Mifune}
\address[Mifune]{Graduate School of Mathematics, Nagoya University, Furocho, Chikusaku, Nagoya 464-8602, Japan}
\email{yuki.mifune.c9@math.nagoya-u.ac.jp}
\author{Ryo Takahashi}
\address[Takahashi]{Graduate School of Mathematics, Nagoya University, Furocho, Chikusaku, Nagoya 464-8602, Japan}
\email{takahashi@math.nagoya-u.ac.jp}
\urladdr{https://www.math.nagoya-u.ac.jp/~takahashi/}
\thanks{Takahashi was partly supported by JSPS Grant-in-Aid for Scientific Research 23K03070}
\begin{abstract}
Let $R$ be a commutative noetherian local ring with residue field $k$.
Denote by $\db(R)$ the bounded derived category of finitely generated $R$-modules.
In this paper, we study the structure of the Verdier quotient $\db(R)/\thick(R\oplus k)$.
We give necessary and sufficient conditions for it to admit an additive generator.
\end{abstract}
\maketitle
\section{Introduction}

Let $R$ be a commutative noetherian local ring with residue field $k$.
Denote by $\mod R$ the category of finitely generated $R$-modules, by $\db(R)$ the bounded derived category of $\mod R$, and by $\ds(R)=\db(R)/\thick R$ the singularity category of $R$.
In the present paper, we investigate the structure of the Verdier quotient
$$
\d(R)=\db(R)/\thick(R\oplus k)=\ds(R)/\thick k,
$$
which is nothing other than the Verdier quotient of $\db(R)$ by the complexes locally perfect on the punctured spectrum.
Takahashi \cite{dlr} gives a complete classification of the thick subcategories of the triangulated category $\d(R)$ under some mild assumtions.
For instance, if $R$ is locally a hypersurface on the punctured spectrum, then the thick subcategories of the triangulated category $\d(R)$ are parametrized by nonempty specialization-closed subsets of the singular locus of $R$.
Therefore, the category $\d(R)$ is thought of as relatively small.
From this point of view, it is reasonable to ask when $\d(R)$ is smallest as an additive category, namely, when it has an additive generator.
An answer is when $R$ is a Cohen--Macaulay ring of finite CM type, or more generally, when $R$ is an isolated singularity; in this case $\d(R)=0$.
The following theorem provides another answer.

\begin{thm}[Buchweitz--Greuel--Schreyer, Araya--Iima--Takahashi, Kobayashi--Lyle--Takahashi]\label{17}
Let $R$ be a complete equicharacteristic local ring with uncountable algebraically closed residue field of characteristic different from two.
Suppose that the ring $R$ is a hypersurface of countable CM type.
Then $R$ locally has finite CM type on the punctured spectrum, and the triangulated category $\d(R)$ possesses an additive generator.
\end{thm}

Let us explain this theorem.
Let $R$ be as in the theorem; assume $R$ has infinite CM type.
Buchweitz, Greuel and Schreyer \cite{BGS} prove that $R$ is either an $(\A_\infty)$-singularity or a $(\D_\infty)$-singularity, and classify the isomorphism classes of indecomposable maximal Cohen--Macaulay $R$-modules.
Using this classification, Araya, Iima and Takahashi \cite{hsccm} prove that $R$ has {\em finite CM$_+$ type} in the sense of Kobayashi, Lyle and Takahashi \cite{plus}, i.e., there exist only finitely many isomorphism classes of indecomposable maximal Cohen--Macaulay $R$-modules that are not locally free on the punctured spectrum.
Then $\d(R)$ possesses an additive generator (by Lemma \ref{6}).
On the other hand, Kobayashi, Lyle and Takahashi \cite{plus} prove in general that if a Cohen--Macaulay local ring $R$ with a canonical module has finite CM$_+$ type, then it locally has finite CM type on the punctured spectrum.

The main result of the present paper is the following theorem, which describes the relationship of those two conditions which appear in the above theorem: the condition that the ring $R$ locally has finite CM type on the punctured spectrum, and the condition that the category $\d(R)$ possesses an additive generator.

\begin{thm}[Theorems \ref{7} and \ref{4}]\label{1}
Let $R$ be a local ring whose regular locus is Zariski-open.
Suppose that $R$ is locally a Cohen--Macaulay ring of finite CM type on the punctured spectrum.
Then $\d(R)$ possesses an additive generator.
The converse holds true as well, if $R$ is locally Gorenstein on the punctured spectrum.
\end{thm}

As to the last assertion of the above theorem, the assumption of being locally Gorenstein on the punctured spectrum cannot be removed; we shall construct in Example \ref{18} a concrete example which shows it.

Applying the above theorem, we obtain some sufficient conditions for $\d(R)$ to admit an additive generator.

\begin{cor}[Corollaries \ref{19}, \ref{14'} and \ref{14}]\label{2}
\begin{enumerate}[\rm(1)]
\item
Let $R$ be a local ring with Zariski-open regular locus.
If $R$ is locally a simple singularity on the punctured spectrum, then $\d(R)$ admits an additive generator.
\item
Let $R$ be a local ring of Krull dimension one.
If $R$ is locally a hypersurface on the punctured spectrum (e.g., if $R$ is itself a hypersurface), then the triangulated category $\d(R)$ possesses an additive generator.
\item
Let $R$ be a Cohen--Macaulay local ring of Krull dimension one.
Suppose that $R$ has uncountable residue field, and admits a canonical module.
If $R$ is of countable CM type, then $\d(R)$ has an additive generator.
\end{enumerate}
\end{cor}

Here, a {\em simple singularity} means a commutative noetherian local ring $R$ with residue field of characteristic zero such that the complete tensor product of the completion of $R$ with the algebraic closure of $k$ is either an $(\A_n)$-singularity ($n\ge1$) or a $(\D_n)$-singularity ($n\ge4$) or an $(\E_n)$-singularity ($n=6,7,8$).
As for the second assertion of the corollary, we can actually construct an explicit additive generator of $\d(R)$; see Remark \ref{21}.

This paper is organized as follows.
In Section 2, we study additive closures of syzygies to show Lemma \ref{13}.
In Section 3, we explore the J-1 property of a local ring, and the property of being locally of finite CM type on the punctured spectrum to prove Lemma \ref{10}.
We also investigate basic properties of $\d(R)$.
In Section 4, we consider when $\d(R)$ has an additive generator.
We prove Theorems \ref{7} and \ref{4} by using results in the previous sections including Lemmas \ref{13} and \ref{10}.
In Section 5, applying Theorem \ref{4} to the cases of locally a simple singularity on the punctured spectrum and of dimension one, we obtain Corollaries \ref{19}, \ref{14'} and \ref{14}.

\section{Additive closures, syzygies and indecomposables}

In this section, we investigate several fundamental properties of additive closures, syzygies, and indecomposable modules.
We then state a certain property of additive closures of syzygies, which is necessary in a later section.
We begin with our convention, which is valid throughout the paper.

\begin{conv}
We assume that all rings are commutative and noetherian, all modules are finitely generated, and all subcategories are nonempty and strictly full.
Unless otherwise specified, $R$ is a local ring of (Krull) dimension $d$ with maximal ideal $\m$ and residue field $k$.
We denote by $\mod R$ the category of (finitely generated) $R$-modules.
For a prime ideal $\p$ of $R$, we denote by $\kappa(\p)$ the residue field of the local ring $R_\p$, that is to say, $\kappa(\p)=R_\p/\p R_\p$.
For an $R$-module $M$, we denote by $[M]$ the isomorphism class of $M$.
Subscripts and superscripts may be omitted if they are clear from the context.
\end{conv}

We state the definition of additive closures and some of their basic properties.

\begin{dfn}
Let $\C$ be an additive category.
For a subcategory $\X$ of $\C$, we denote by $\add_\C\X$ the {\em additive closure} of $\X$ in $\C$.
This is defined to be the smallest subcategory of $\C$ which contains $\X$ and is closed under finite direct sums and direct summands.
When $\X$ consists of a single object $M$, we set $\add_\C M=\add_\C\X$.
An {\em additive generator} $G$ of $\C$ is defined as an object of $\C$ such that $\add_\C G=\C$.
\end{dfn}

\begin{rem}
\begin{enumerate}[(1)]
\item
The additive closure $\add_{\mod R}R$ consists of the free $R$-modules.
\item
Let $\C$ be an additive category, and let $\X$ be a subcategory of $\C$.
Then the following statements hold.
\begin{enumerate}[(a)]
\item
An object $C$ of $\C$ belongs to $\add_\C\X$ if and only if there exist finitely many objects $X_1,\dots,X_n$ in $\X$ such that $C$ is a direct summand of the direct sum $X_1\oplus\cdots\oplus X_n$.
\item
One has $\X=\add_\C\X$ if and only if $\X$ is closed under finite direct sums and direct summands.
\item
Let $\Y$ be another subcategory of $\C$.
If one has $\X\subseteq\Y$, then one has $\add_\C\X\subseteq\add_\C\Y$.
\end{enumerate}
\end{enumerate}
\end{rem}

Next, we give the definition and several properties of syzygies.

\begin{dfn}
\begin{enumerate}[(1)]
\item
For an $R$-module $M$ and an integer $n\ge0$, we denote by $\syz^n_RM$ the {\em $n$th syzygy} of $M$, that is, the image of the $n$th differential map in a minimal free resolution of $M$.
(By definition $\syz^0M=M$.)
\item
Let $\X$ be a subcategory of $\mod R$.
For a nonnegative integer $n$, we denote by $\syz^n_R\X$ the subcategory of $\mod R$ consisting of $R$-modules $M$ such that there exists an exact sequence $0\to M\to F_{n-1}\to\cdots\to F_1\to F_0\to X\to0$ of $R$-modules with $X\in\X$ and $F_0,F_1,\dots,F_{n-1}\in\add R$.
(By definition $\syz^0\X=\X$.)
\end{enumerate}
\end{dfn}

\begin{rem}
\begin{enumerate}[(1)]
\item
Let $M$ be an $R$-module, and let $n$ be a nonnegative integer.
The $n$th syzygy $\syz^nM$ of $M$ is uniquely determined up to isomorphism, since so is a minimal free resolution of the $R$-module $M$.
\item
Let $\X\subseteq\Y$ be subcategories of $\mod R$.
One then has $\syz^n\X\subseteq\syz^n\Y$ for each nonnegative integer $n$.
\item
Let $\X$ be a subcategory of $\mod R$.
If $\X$ is closed under finite direct sums, then so is $\syz^n\X$ for each $n\ge0$.
\item
Let $\X$ be a subcategory of $\mod R$ and let $n>0$.
Then one has $\syz^n\X=\{\syz^nX\oplus R^{\oplus m}\,|\, X\in\X,\,m\ge0\}$.
In particular, the inclusion $\add R\subseteq\syz^n\X$ holds when $\X$ contains the zero module $0$, since $\syz^n0=0$.
\end{enumerate}
\end{rem}

Now, for a given subcategory, we study isomorphism classes of indecomposable modules that belong to it.

\begin{dfn}
Let $\X$ be a subcategory of $\mod R$.
Then the set of isomorphism classes of indecomposable $R$-modules belonging to $\X$ is denoted by $\ind\X$, namely, one has $\ind\X=\{[X]\,|\,X\in\X\}$.
\end{dfn}

\begin{rem}
Let $\X$ and $\Y$ be subcategories of $\mod R$.
Then the following two statements hold.
\begin{enumerate}[(1)]
\item
If there is an inclusion $\X\subseteq\Y$ of subcategories, then there is an inclusion $\ind\X\subseteq\ind\Y$ of sets.
\item
Suppose that the subcategories $\X$ and $\Y$ are closed under finite direct sums and direct summands.
Then the converse of (1) holds: $\ind\X\subseteq\ind\Y$ implies $\X\subseteq\Y$.
Indeed, let $X\in\X$.
Decompose $X\cong X_1\oplus\cdots\oplus X_n$ into indecomposable $R$-modules.
As $\X$ is closed under direct summands, each $X_i$ belongs to $\X$, so that it belongs to $\Y$.
Since $\Y$ is closed under finite direct sums, the module $X$ belongs to $\Y$. 
\end{enumerate}
\end{rem}

The lemma below plays an important role in the proof of one of the main results of this paper.

\begin{lem}\label{13}
Let $X$ be an $R$-module, and let $n\ge0$ be an integer.
Suppose that the inclusion $\syz^n(\mod R)\subseteq\add X$ holds.
Then there exists an integer $r\ge0$ such that the equality $\add(\syz^{in}X)=\add(\syz^{(i+1)n}X)$ holds for all integers $i\ge r$.
As a consequence, one has that $\add(R\oplus\syz^{in}X)=\add(R\oplus\syz^{(i+1)n}X)$ for all integers $i\ge r$.
\end{lem}

\begin{proof}
Since $\syz^nX\in\syz^n(\mod R)\subseteq\add X$, we get a descending chain $\add X\supseteq\add\syz^nX\supseteq\add\syz^{2n}X\supseteq\cdots$ of subcategories of $\mod R$, which induces a descending chain $\ind(\add X)\supseteq\ind(\add\syz^nX)\supseteq\ind(\add\syz^{2n}X)\supseteq\cdots$ of sets.
By \cite[Theorem 2.2]{LW}, the set $\ind(\add X)$ is finite.
The latter descending chain stabilizes: there exists $r\ge0$ with $\ind(\add\syz^{in}X)=\ind(\add\syz^{(i+1)n}X)$ for all $i\ge r$.
Hence $\add\syz^{in}X=\add\syz^{(i+1)n}X$ for all $i\ge r$.
\end{proof}

\section{Finite CM type on the punctured spectrum and the triangulated category $\d(R)$}

In this section, we first investigate the J-1 property of a local ring, and relate it to the property of being locally of finite CM type on the punctured spectrum.
Then we give the definition of the triangulated category $\d(R)$ and explore its fundamental properties.
Let us begin with recalling several notions.

\begin{dfn}
\begin{enumerate}[(1)]
\item
The {\em punctured spectrum} of a local ring $(R,\m)$ is defined as the set $\spec R\setminus\{\m\}$. 
\item
Let $\P$ be a property of local rings.
It is said that $R$ {\em locally satisfies $\P$ on the punctured spectrum} if the local ring $R_\p$ satisfies $\P$ for any $\p\in\spec R\setminus\{\m\}$.
We similarly define the module and complex versions. 
\item
Denote by $\reg R$ and $\sing R$ the {\em regular locus} and the {\em singular locus} of $R$, respectively.
Namely, $\reg R$ is the set of prime ideals $\p$ such that $R_\p$ is regular, and $\sing R$ is the complement of $\reg R$ in $\spec R$.
\item
Following \cite[(32.B)]{M}, we say that $R$ is {\em J-1} if $\reg R$ is an open subset of $\spec R$ in the Zariski topology.
Needless to say, this is equivalent to the condition that $\sing R$ is a closed subset of $\spec R$.
\item
A {\em specialization-closed} subset of $\spec R$ is by definition a subset $W$ of $\spec R$ such that if $\p$ is a prime ideal of $R$ that belongs to the set $W$ and $\q$ is a prime ideal of $R$ that contains $\p$, then $\q$ belongs to $W$.
\item
Let $\Phi$ be a subset of $\spec R$.
The {\em (Krull) dimension} $\dim\Phi$ of $\Phi$ is defined to be the supremum of integers $n\ge0$ such that there exists a chain $\p_0\subsetneq\p_1\subsetneq\cdots\subsetneq\p_n$ of prime ideals of $R$ which belong to the set $\Phi$.
\item
We say that $R$ is an {\em isolated singularity} if $R$ is locally regular on the punctured spectrum, or equivalently, if $\sing R\subseteq\{\m\}$, or equivalently, if $\sing R$ has dimension at most zero (i.e., equal to $0$ or $-\infty$).
\end{enumerate}
\end{dfn}

\begin{rem}
\begin{enumerate}[(1)]
\item
The singular locus $\sing R$ of $R$ is always a specialization-closed subset of $\spec R$.
\item
A subset of $\spec R$ is specialization-closed if and only if it is a (possibly infinite) union of closed subsets of $\spec R$ in the Zariski topology.
In particular, every closed subset of $\spec R$ is specialization-closed.
\item
Let $W$ be a specialization-closed subset of $\spec R$.
Then one has $\dim W=\sup\{\dim R/\p\,|\,\p\in W\}$.
\item
A local ring $R$ is regular if and only if the equality $\sing R=\emptyset$ holds, if and only if the equality $\dim\sing R=-\infty$ holds.
In particular, in our sense, any regular local ring is an isolated singularity.
\end{enumerate}
\end{rem}

If $\dim R\le1$, then $\spec R$ consists of the minimal prime ideals and the maximal ideal of $R$, whence $\spec R$ is finite.
The first part of the lemma below thus follows, whose second part is easily deduced from the first.

\begin{lem}\label{16}
Let $R$ be a local ring of dimension at most one.
Then $\spec R$ is a finite set, and $R$ is J-1.
\end{lem}

Now we move on to studying maximal Cohen--Macaulay modules.
To begin with, we recall the definition.

\begin{dfn}
\begin{enumerate}[(1)]
\item
We say that an $R$-module $M$ is {\em maximal Cohen--Macaulay} if either $M=0$ or $\depth M=\dim R$.
Denote by $\cm(R)$ the subcategory of $\mod R$ consisting of maximal Cohen--Macaulay $R$-modules.
\item
We say that a local ring $R$ has {\em finite CM type} provided that $\ind\cm(R)$ is a finite set.
\end{enumerate}
\end{dfn}

\begin{rem}
\begin{enumerate}[(1)]
\item
Since the zero module is a maximal Cohen--Macaulay module by definition, $\cm(R)$ is an additive category, or to be more precise, $\cm(R)$ is an additive subcategory of $\mod R$.
\item
The subcategory $\cm(R)$ of $\mod R$ is closed under finite direct sums and direct summands.
If the ring $R$ is Cohen--Macaulay, then $\cm(R)$ contains $R$, and is closed under syzygies as a subcategory of $\mod R$.
\end{enumerate}
\end{rem}

We provide a useful lemma on finite CM type, which would be well-known to experts.

\begin{lem}\label{8}
A local ring $R$ has finite CM type if and only if the category $\cm(R)$ has an additive generator.
\end{lem}

\begin{proof}
By \cite[Theorem 2.2]{LW} the set $\ind(\add G)$ is finite for an $R$-module $G$.
This shows the ``if'' part. 
To prove the ``only if'' part, assume $\ind\cm(R)$ is finite.
There exist an integer $n\ge0$ and maximal Cohen--Macaulay $R$-modules $G_1,\dots,G_n$ with $\ind\cm(R)=\{[G_1],\dots,[G_n]\}$.
Put $G=G_1\oplus\cdots\oplus G_n$.
As $\cm(R)$ contains $G$ and is closed under finite direct sums and direct summands, $\cm(R)$ contains $\add G$.
Let $M$ be a maximal Cohen--Macaulay $R$-module, and take a decomposition $M\cong M_1\oplus\cdots\oplus M_m$ of $M$ into indecomposable $R$-modules.
Then for each $1\le i\le m$ the isomorphism class $[M_i]$ of $M_i$ belongs to $\ind\cm(R)$, so that $M_i\cong G_{l_i}$ for some $1\le l_i\le n$.
We have $M\cong G_{l_1}\oplus\cdots\oplus G_{l_m}\in\add G$.
We thus obtain a desired equality $\cm(R)=\add G$.
\end{proof}

The condition that $R$ is locally of finite CM type on the punctured spectrum plays an essential role in this paper.
The following lemma describes the relationship of this condition with the size of the singular locus.

\begin{lem}\label{10}
Suppose that $R$ is locally a Cohen--Macaulay ring of finite CM type on the punctured spectrum.
Then the singular locus $\sing R$ has dimension at most one.
Hence, $\sing R$ is a finite set if the ring $R$ is J-1.
\end{lem}

\begin{proof}
Suppose that $\sing R$ has dimension at least two.
Then there exists a chain $\p\subsetneq\q\subsetneq\m$ in $\sing R$.
By assumption, $R_\q$ is a Cohen--Macaulay local ring of finite CM type.
By \cite[Theorem 7.12]{LW}, the localization $R_\q$ is an isolated singularity.
Therefore, $R_\p=(R_\q)_{\p R_\q}$ is a regular local ring, but this contradicts the choice of $\p$.
Thus the set $\sing R$ has dimension at most one.
Now assume that $R$ is J-1.
Then we have $\sing R=\V(I)$ for some ideal $I$ of $R$, and see that $\dim R/I\le1$.
Lemma \ref{16} implies that $\spec R/I$ is finite, and so is $\sing R$.
\end{proof}

Now we shall give the definition of the triangulated category $\d(R)$, which is the main target of this paper.

\begin{dfn}
\begin{enumerate}[(1)]
\item
Let $\T$ be a triangulated category.
A {\em thick subcateory} of $\T$ is defined as a triangulated subcategory of $\T$ which is closed under direct summands.
For an object $G$ of $\T$, we denote by $\thick_\T G$ the {\em thick closure} of $G$ in $\T$, which is defined to be the smallest thick subcategory of $\T$ containing $G$.
\item
Denote by $\db(R)$ the bounded derived category of $\mod R$.
Let $\ds(R)$ stand for the {\em singularity category} of $R$, which is the Verdier quotient of $\db(R)$ by the perfect complexes, i.e., $\ds(R)=\db(R)/\thick R$.
(Recall that a {\em perfect complex} is by definition a bounded complex of finitely generated projective modules.)
\item
We define the triangulated category $\d(R)$ to be the Verdier quotient $\db(R)/\thick(R\oplus k)$ of the triangulated category $\db(R)$, which is the same as the Verdier quotient $\ds(R)/\thick k$ of $\ds(R)$.
\end{enumerate}
\end{dfn}

In the rest of this paper we often use the lemma below as a basic tool for computation of objects in $\d(R)$.

\begin{lem}\label{6}
\begin{enumerate}[\rm(1)]
\item
For every object $X$ of $\db(R)$, there exists an exact triangle $F\to X\to M[n]\rightsquigarrow$ in $\db(R)$ such that $F\in\thick_{\db(R)}R$, $M\in\mod R$ and $n\in\Z$.
As a consequence, for every object $X$ of $\d(R)$ there exist an $R$-module $M$ and an integer $n$ such that one has an isomorphism $X\cong M[n]$ in $\d(R)$.  
\item
Let $M\in\mod R$ and $n\in\Z$.
Assume that $n\ge0$.
Then there exists an exact triangle $F\to M\to\syz^nM[n]\rightsquigarrow$ in $\db(R)$ such that $F$ is in $\thick_{\db(R)}R$.
Consequently, one has an isomorphism $M[-n]\cong\syz^nM$ in $\d(R)$.
\item
Let $X$ be an object of $\db(R)$.
Then there is an isomorphism $X\cong0$ in the triangulated category $\d(R)$ if and only if the $R$-complex $X$ locally has finite projective dimension on the punctured spectrum, that is to say, the $R_\p$-complex $X_\p$ is isomorphic in $\db(R_\p)$ to a perfect $R_\p$-complex for every $\p\in\spec R\setminus\{\m\}$.
\end{enumerate}
\end{lem}

\begin{proof}
Assertions (1) and (2) follow from \cite[Lemma 2.4]{sing} for instance.
Let us show assertion (3).
An object $X\in\db(R)$ is isomorphic to $0$ in $\d(R)$ if and only if $X$ is in $\thick_{\db(R)}(R\oplus k)$.
By \cite[Corollary 4.3(3)]{kos}, this thick closure consists of the $R$-complexes locally of finite projective dimension on the punctured spectrum.
\end{proof}

As a direct application of the above lemma, we obtain a characterization of isolated singularities.

\begin{prop}\label{11}
One has that $\d(R)=0$ if and only if the local ring $R$ is an isolated singularity.
\end{prop}

\begin{proof}
In view of Lemma \ref{6}(3), it is enough to verify that the local ring $R$ is an isolated singularity if and only if for each object $X$ of $\db(R)$ and each nonmaximal prime ideal $\p$ of $R$ the $R_\p$-complex $X_\p$ has finite projective dimension.
The ``only if'' part follows from the Auslander--Buchsbaum--Serre theorem and Lemma \ref{6}(1).
The ``if'' part is shown by taking $X=R/\p$ and using the Auslander--Buchsbaum--Serre theorem.
\end{proof}

\section{Necessary and sufficient conditions for $\d(R)$ to admit an additive generator}

In this section, we shall state and prove the main theorems of this paper, which give necessary and sufficient conditions for our triangulated category $\d(R)$ to have an additive generator.
In the proof of the first main theorem, we need the following general fundamental fact about Verdier quotients of triangulated categories.

\begin{lem}\label{3}
Let $\T$ be a triangulated category, and $\U$ a thick subcategory of $\T$.
Let $X,Y\in\T$.
Then $X\cong Y$ in $\T/\U$ if and only if there exist exact triangles $E\to X\to A\rightsquigarrow$ and $E\to Y\to B\rightsquigarrow$ in $\T$ with $A,B\in\U$. 
\end{lem}

\begin{proof}
The ``if'' part is obvious.
To prove the ``only if'' part, we let $(X\xleftarrow{s}E\xrightarrow{a}Y)$ be an isomorphism in the Verdier quotient $\T/\U$, where $s$ and $a$ are morphisms in $\T$ and $A:=\cone(s)$ belongs to $\U$.
An exact triangle $E\xrightarrow{s}X\to A\rightsquigarrow$ in $\T$ is induced.
Since $(E\xleftarrow{1}E\xrightarrow{a}Y)$ is an isomorphism in $\T/\U$, the object $B:=\cone(a)$ belongs to $\U$; see \cite[2.1.23 and 2.1.35]{N}.
An exact triangle $E\xrightarrow{a}Y\to B\rightsquigarrow$ in $\T$ is induced.
We are done.
\end{proof}

Next we recall the definition of the stable category of maximal Cohen--Macaulay modules and a fact on it.

\begin{dfn}
Let $R$ be a Cohen--Macaulay local ring.
Let $\lcm(R)$ be the {\em stable category} of $\cm(R)$.
This category is defined as follows: the objects of $\lcm(R)$ are the same as those of $\cm(R)$, and the hom-set from $M$ to $N$ is given by the quotient of $\Hom_R(M,N)$ by the homomorphisms factoring through free $R$-modules.
\end{dfn}

\begin{rem}
Assume $R$ is a Gorenstein local ring.
Then $\lcm(R)$ is a triangulated category, and there exists a triangle equivalence $\lcm(R)\cong\ds(R)$.
This is a celebrated result due to Buchweitz \cite[Theorem 4.4.1]{B}.
\end{rem}

We are ready to give a necessary condition for the category $\d(R)$ to possess an additive generator.

\begin{thm}\label{7}
Let $R$ be a local ring which is locally Gorenstein on the punctured spectrum.
If the triangulated category $\d(R)$ admits an additive generator, then $R$ locally has finite CM type on the punctured spectrum.
\end{thm}

\begin{proof}
By assumption, we have $\d(R)=\add G$ for some $G\in\d(R)$.
Fix a nonmaximal prime ideal $\p$ of $R$.

We claim that $\ds(R_\p)=\add G_\p$.
Indeed, let $X\in\db(R_\p)$.
Then there exists $Y\in\db(R)$ such that $Y_\p\cong X$ in $\db(R_\p)$; see \cite[Lemma 4.2]{ddc} for instance. 
We find an object $Z\in\db(R)$ and a nonnegative integer $n$ such that $Y\oplus Z\cong G^{\oplus n}$ in $\d(R)$.
Lemma \ref{3} gives rise to exact triangles $E\to Y\oplus Z\to A\rightsquigarrow$ and $E\to G^{\oplus n}\to B\rightsquigarrow$ in $\db(R)$ such that $A,B\in\thick_{\db(R)}(R\oplus k)$.
Localization at $\p$ yields exact triangles $E_\p\to Y_\p\oplus Z_\p\to A_\p\rightsquigarrow$ and $E_\p\to G_\p^{\oplus n}\to B_\p\rightsquigarrow$ in $\db(R_\p)$.
Note here that $A_\p,B_\p\in\thick_{\db(R_\p)}(R_\p\oplus k_\p)=\thick_{\db(R_\p)}R_\p$.
Consequently, the object $Y_\p\oplus Z_\p$ is isomorphic to $G_\p^{\oplus n}$ in $\ds(R_\p)$.
Thus $X$ belongs to $\add_{\ds(R_\p)}G_\p$, and the claim follows.

Since $R_\p$ is Gorenstein, there exists an equivalence $\ds(R_\p)\cong\lcm(R_\p)$ of categories.
Let $H\in\lcm(R_\p)$ be the object that corresponds via this equivalence to $G_\p\in\ds(R_\p)$.
The above claim shows $\lcm(R_\p)=\add H$, which implies that $\cm(R_\p)=\add(R_\p\oplus H)$.
From Lemma \ref{8} we deduce that $R_\p$ has finite CM type.
\end{proof}

In the proof of our next theorem, it is necessary for us to use the notion of nonfree loci for modules.

\begin{dfn}
For an $R$-module $M$, we denote by $\nf(M)$ the {\em nonfree locus} of $M$, which is defined to be the set of prime ideals $\p$ of $R$ such that the $R_\p$-module $M_\p$ is nonfree.
\end{dfn}

\begin{rem}
\begin{enumerate}[(1)]
\item
For every $R$-module $M$ the subset $\nf(M)$ of $\spec R$ is closed in the Zariski topology.
\item
Let $M$ be an $R$-module.
Then $\nf(\syz^iM)$ is contained in $\sing R$ for all integers $i\ge d$.
This is an immediate consequence of the Auslander--Buchsbaum--Serre theorem and the Auslander--Buchsbaum formula.
\end{enumerate}
\end{rem}

Now we can show the theorem below giving a sufficient condition for $\d(R)$ to have an additive generator.

\begin{thm}\label{4}
Let $R$ be a J-1 local ring.
Suppose that $R$ is locally a Cohen--Macaulay ring of finite CM type on the punctured spectrum.
Then the triangulated category $\d(R)$ possesses an additive generator.
\end{thm}

\begin{proof}
If $R$ is an isolated singularity, Proposition \ref{11} implies $\d(R)=0=\add0$.
Assume $R$ is not so.
Apply Lemma \ref{10} to write $\sing R=\{\p_1,\dots,\p_n,\m\}$ with $n\ge1$.
By assumption, $\cm(R_{\p_i})$ has an additive generator $Z_i$ for every $1\le i\le n$.
Choose an $R$-module $Y_i$ such that $Z_i=(Y_i)_{\p_i}$.
Put $Y=Y_1\oplus\cdots\oplus Y_n$.
We see that
\begin{equation}\label{24}
\syz^d(\mod R_{\p_i})\subseteq\cm(R_{\p_i})=\add Z_i=\add(Y_i)_{\p_i}\subseteq\add Y_{\p_i}\quad\text{for each }1\le i\le n.
\end{equation}
By virtue of Lemma \ref{13}, there exists an integer $r\ge1$ such that the following statement holds.
\begin{equation}\label{25}
\add(R_{\p_i}\oplus\syz^{rd}(Y_{\p_i}))=\add(R_{\p_i}\oplus\syz^{(r+1)d}(Y_{\p_i}))\quad\text{for each }1\le i\le n.
\end{equation}
Put $G=\syz^{rd}Y$ and $\Phi=\{\p_1,\dots,\p_n\}$.
The set $\Phi$ is not empty but finite.
We proceed step by step.

(i) Let $M\in\syz^{(r+1)d}(\mod R)$.
Set $\X=\add_{\mod R}(R\oplus G)$.
Using \eqref{24}, for every $1\le i\le n$ we have
$$
M_{\p_i}\in\syz^{(r+1)d}(\mod R_{\p_i})=\syz^{rd}\syz^d(\mod R_{\p_i})\subseteq\syz^{rd}(\add Y_{\p_i})\subseteq\add(R_{\p_i}\oplus G_{\p_i})\subseteq\add_{\mod R_{\p_i}}(\X_{\p_i}).
$$
The subcategory $\X$ of $\mod R$ is closed under finite direct sums and contains $R$.
Applying \cite[Lemma 3.7(2)]{dlr}, we get an exact sequence $0\to L\to M\oplus N\to X\to0$ of $R$-modules such that $X\in\X$, $\nf(L)\subseteq\nf(M)$ and $\nf(L)\cap\Phi=\emptyset$.
As $(r+1)d\ge d$, we have $\nf(M)\subseteq\sing R=\Phi\cup\{\m\}$.
Hence $\nf(L)$ is contained in $\{\m\}$, and $L\cong0$ in $\d(R)$ by Lemma \ref{6}(3).
The induced isomorphism $M\oplus N\cong X$ in $\d(R)$ shows $M\in\add_{\d(R)}G$.

(ii) Put $\X=\add_{\mod R}(R\oplus\syz^dG)$.
It is observed from \eqref{25} that $G_{\p_i}\in\add_{\mod R_{\p_i}}(\X_{\p_i})$ for all $1\le i\le n$.
Similarly as in (i), we get an exact sequence $0\to L\to G\oplus N\to X\to0$ with $X\in\X$, $\nf(L)\subseteq\nf(G)$ and $\nf(L)\cap\Phi=\emptyset$.
As $rd\ge d$, we have $\nf(G)\subseteq\sing R=\Phi\cup\{\m\}$.
Thus $\nf(L)\subseteq\{\m\}$ and $L\cong0$ in $\d(R)$.
The isomorphism $G\oplus N\cong X$ in $\d(R)$ says $G\in\add_{\d(R)}(\syz^dG)=\add_{\d(R)}G[-d]$ by Lemma \ref{6}(2).
We have $G[d]\in\add_{\d(R)}G$, from which we see that $G[id]\in\add_{\d(R)}G$ for all $i\ge0$.
As $R$ is not an isolated singularity, we have $d\ge1$.
The module $\syz^{(r+1)d+(d-1)}G$ is in $\syz^{(r+1)d}(\mod R)$, and it belongs to $\add_{\d(R)}G$ by (i).
Hence $G[1]=(\syz^{(r+1)d+(d-1)}G)[(r+2)d]\in\add_{\d(R)}G[(r+2)d]\subseteq\add_{\d(R)}G$.
Thus, $G[i]$ is in $\add_{\d(R)}G$ for all $i\ge0$.

(iii) Let $M\in\syz^{(r+1)d}(\mod R)$ and $i\ge0$.
Lemma \ref{6}(2) shows $M[-i]\cong\syz^iM$ in $\d(R)$.
As $\syz^iM$ belongs to $\syz^{(r+1)d}(\mod R)$ as well, it is in $\add_{\d(R)}G$ by (i).
Hence $M[-i]\in\add_{\d(R)}G$.
Since $M$ is itself in $\add_{\d(R)}G$ by (i) again, $M[i]$ belongs to $\add_{\d(R)}G[i]$, while this additive closure is contained in $\add_{\d(R)}G$ by (ii).
We conclude that $M[j]$ is in $\add_{\d(R)}G$ for every integer $j$ and every $R$-module $M$ that belongs to $\syz^{(r+1)d}(\mod R)$.

(iv) Let $X\in\db(R)$.
By Lemma \ref{6}(1)(2) we have $X\cong(\syz^{(r+1)d}M)[n]$ in $\d(R)$ for some $M\in\mod R$ and $n\in\Z$.
As $(\syz^{(r+1)d}M)[n]\in\add_{\d(R)}G$ by (iii), we get $X\in\add_{\d(R)}G$.
We have shown $\d(R)=\add_{\d(R)}G$.
\end{proof}

\begin{rem}
If $R$ is a Cohen--Macaulay local ring which is locally Gorenstein on the punctured spectrum, then Theorem \ref{4} can be proved in the following approach.
As is shown below, the following statement holds.  
\begin{equation}\label{27}
\text{For every $M\in\cm(R)$ and every $n\in\Z$ there exists $N\in\cm(R)$ such that $M[n]\cong N$ in $\d(R)$.}
\end{equation}
Let $X\in\d(R)$.
Then $X\cong(\syz^dM)[n]$ in $\d(R)$ for some $M\in\mod R$ and $n\in\Z$ by Lemma \ref{6}(1)(2).
As $\syz^dM$ is in $\cm(R)$, by \eqref{27} we find $N\in\cm(R)$ with $(\syz^dM)[n+(r+1)d]\cong N$.
Hence $X\cong N[-(r+1)d]\cong\syz^{(r+1)d}N$ in $\d(R)$ and $\syz^{(r+1)d}N$ belongs to $\add_{\d(R)}G$ by (i) in the proof of Theorem \ref{4}.
Thus $\d(R)=\add G$ holds.

Let us show \eqref{27}.
Fix $m\ge0$.
Lemma \ref{6}(2) says that $M[-m]\cong\syz^mM$ in $\d(R)$, and $\syz^mM$ is a maximal Cohen--Macaulay $R$-module.
Therefore, it suffices to prove that $M[m]\cong N$ in $\d(R)$ for some $N\in\cm(R)$.

If $R$ admits a canonical module $\omega$, we can do this simply as follows.
There is an exact sequence $0\to M\to E^0\to\cdots\to E^{m-1}\to N\to0$ of maximal Cohen--Macaulay $R$-modules with $E^i\in\add\omega$ for any $0\le i\le m-1$.
As $R$ is locally Gorenstein on the punctured spectrum, $\omega$ is locally free on the punctured spectrum.
Lemma \ref{6}(3) implies that $E^i\cong0$ in $\d(R)$ for every $0\le i\le m-1$.
It is easy to observe that $M[m]\cong N$ in $\d(R)$.

Now we consider the general case.
As $R$ is locally Gorenstein on the punctured spectrum, by \cite[Proposition 12.8]{LW} there is an exact sequence $0\to M\to F_{d-1}\to\cdots\to F_0\to C\to0$ in $\mod R$ with each $F_i$ free.
Lemma \ref{6}(2) implies $M\cong C[-d]$ in $\d(R)$.
Let $L=\Gamma_\m(C)$ be the $\m$-torsion submodule of $C$, and set $K=C/L$.
Since $L$ has finite length as an $R$-module, we have that $L\cong0$ in $\d(R)$ by Lemma \ref{6}(3), and hence $C\cong K$ in $\d(R)$.
Since the $R$-module $K$ has positive depth, the depth lemma implies that $H:=\syz^{d-1}K$ is in $\cm(R)$ (note that $d\ge1$ as $R$ is not an isolated singularity).
It holds that $M[1]\cong C[1-d]\cong K[1-d]\cong H$ in $\d(R)$.
Applying this argument to $H$ and repeating it, we eventually find $N\in\cm(R)$ such that $M[m]\cong N$ in $\d(R)$. 
\end{rem}

Finally, we consider the necessity of the assumption of Theorem \ref{7}.
We need a lemma.

\begin{lem}\label{5}
\begin{enumerate}[\rm(1)]
\item
An artinian local ring $R$ has finite CM type if and only if $R$ is a hypersurface.
\item
Let $(R,\m,k)$ be a local ring with $\m^2=0$.
For each $R$-module $M$ the first syzygy $\syz M$ is a $k$-vector space.
\end{enumerate}
\end{lem}

\begin{proof}
(1) The assertion immediately follows from \cite[Theorem 3.3]{LW}.

(2) Let $F=(\cdots\to F_1\to F_0\to0)$ be a minimal free resolution of $M$.
Then the inclusion map $\syz M\hookrightarrow F_0$ factors through $\m F_0$.
Since $\m(\m F_0)=\m^2F_0=0$, we have $\m\syz M=0$.
Therefore, $\syz M$ is a $k$-vector space.
\end{proof}

The following example shows that Theorem \ref{7}, which is viewed as the converse of Theorem \ref{4}, does not necessarily hold true without the assumption that $R$ is locally Gorenstein on the punctured spectrum.

\begin{ex}\label{18}
Let $k$ be a field, and consider $R=k[\![x,y,z]\!]/(x^2,xy,y^2)$.
Then $R$ is a $1$-dimensional Cohen--Macaulay complete (hence, J-1) local ring.
We have $\spec R=\{\p,\m\}$, where $\p=(x,y)$ and $\m=(x,y,z)$.
The localization $R_\p=k[\![x,y,z]\!]_{(x,y)}/(x^2,xy,y^2)$ is an artinian non-Gorenstein local ring.
It is observed from Lemma \ref{5}(1) that $R_\p$ does not have finite CM type.
Therefore, the Cohen--Macaulay local ring $R$ is neither locally Gorenstein on the punctured spectrum, nor locally of finite CM type on the punctured spectrum.

We claim that $\d(R)=\add R/\p$ holds.
In fact, we fix any object $X\in\d(R)$.
It follows from (1) and (2) of Lemma \ref{6} that $X\cong M[n]\cong(\syz_RM)[n+1]$ in $\d(R)$ for some $M\in\mod R$ and $n\in\Z$.
Lemma \ref{5}(2) gives an isomorphism $\syz_{R_\p}(M_\p)\cong\kappa(\p)^{\oplus a}$ with $a\ge0$.
Hence $(\syz_RM)_\p\cong R_\p^{\oplus b}\oplus\kappa(\p)^{\oplus a}\cong(R^{\oplus b}\oplus(R/\p)^{\oplus a})_\p$ for some $b\ge0$.
There is an exact sequence $0\to K\to\syz_RM\to R^{\oplus b}\oplus(R/\p)^{\oplus a}\to C\to0$ in $\mod R$ with $K_\p=0=C_\p$.
Hence $K$ and $C$ are locally free on the punctured spectrum of $R$.
Lemma \ref{6}(3) shows $K\cong0\cong C$ in $\d(R)$, which implies that $\syz_RM\cong R^{\oplus b}\oplus(R/\p)^{\oplus a}\cong(R/\p)^{\oplus a}$ in $\d(R)$.
We get an isomorphism $X\cong(R/\p)^{\oplus a}[n+1]$ in $\d(R)$.
It is easy to observe that $\p=(x,y)=(x)\oplus(y)\cong R/(0:x)\oplus R/(0:y)=(R/\p)^{\oplus2}$ in $\mod R$.
There are isomorphisms $R/\p\cong\p[1]\cong(R/\p)[1]^{\oplus2}$ in $\d(R)$.
It is seen that $\add_{\d(R)}R/\p=\add_{\d(R)}((R/\p)[i])$ for all integers $i$.
We obtain $X\cong(R/\p)^{\oplus a}[n+1]\in\add_{\d(R)}((R/\p)[n+1])=\add_{\d(R)}R/\p$.
The claim now follows.
\end{ex}

\section{Several consequences and applications}

In this section, we give several consequences and applications of Theorem \ref{4}.
First of all, we state another sufficient condition for $\d(R)$ to have an additive generator, as a corollary of the proof of Theorem \ref{4}.

\begin{cor}\label{26}
Let $R$ be a J-1 local ring.
Let $G$ be an $R$-module.
Suppose that for each $\m\ne\p\in\sing R$ the local ring $R_\p$ is Cohen--Macaulay and $\cm(R_\p)=\add G_\p$.
Then one has the equality $\d(R)=\add G$.
\end{cor}

\begin{proof}
If $R$ is an isolated singularity, then $\d(R)=0$ by Proposition \ref{11}, and $\d(R)=\add_{\d(R)}G$.
Assume that $R$ is not an isolated singularity.
The local ring $R$ is locally a Cohen--Macaulay ring of finite CM type on the punctured spectrum by Lemma \ref{8}.
We can write $\sing R=\{\p_1,\dots,\p_n,\m\}$ with $n\ge1$ by Lemma \ref{10}.
We have $\syz^d(\mod R_{\p_i})\subseteq\cm(R_{\p_i})=\add G_{\p_i}$ for all integers $1\le i\le n$.
The proof of Theorem \ref{4} gives rise to an integer $m\ge d$ such that $\d(R)=\add(\syz^mG)$.
Using Lemma \ref{6}(2), we get $\d(R)=\add G[-m]=\add G$.
\end{proof}

Next we recall the notions of a hypersurface and a simple singularity.

\begin{dfn}
\begin{enumerate}[(1)]
\item
A (local) {\em hypersurface} is by definition a local ring $(R,\m)$ such that there exist a regular local ring $(S,\n)$ and an element $f\in\n$ such that the $\m$-adic completion $\widehat R$ of $R$ is isomorphic to $S/(f)$.
\item
Let $(R,\m,k)$ be a $d$-dimensional local hypersurface with $\ch k=0$.
Denote by $\overline k$ the algebraic closure of $k$.
We say that $R$ is a {\em simple singularity} if the complete tensor product $\widehat R\widehat\otimes_k\overline k$ is isomorphic to either the formal power series ring $\overline k[\![x_1,\dots,x_d]\!]$ or the complete local hypersurface $\overline k[\![x_0,\dots,x_d]\!]/(f)$ where
$$
f=\begin{cases}
x_0^{n+1}+x_1^2+x_2^2+\cdots+x_d^2&(\A_n)\ (n\in\Z_{\ge1}), \text{ or}\\
x_0^{n-1}+x_0x_1^2+x_2^2+\cdots+x_d^2&(\D_n)\ (n\in\Z_{\ge4}),\text{ or}\\
x_0^4+x_1^3+x_2^2+\cdots+x_d^2&(\E_6),\text{ or}\\
x_0^3x_1+x_1^3+x_2^2+\cdots+x_d^2&(\E_7),\text{ or}\\
x_0^5+x_1^3+x_2^2+\cdots+x_d^2&(\E_8).
\end{cases}
$$
We say that $R$ is a {\em $T$-singularity} when $R$ is a simple singularity whose corresponding polynomial $f$ has type $T\in\{(\A_a),(\D_b),(\E_c)\,|\,a\in\Z_{\ge1},\,b\in\Z_{\ge4},\,c=6,7,8\}$.
\end{enumerate}
\end{dfn}

\begin{rem}
\begin{enumerate}[(1)]
\item
Let $R$ be an artinian local ring.
Then $R$ is a hypersurface if and only if $R$ has embedding dimension at most one, if and only if $R$ is isomorphic to $S/(x^e)$ for some discrete valuation ring $(S,xS)$ and $e>0$.
This follows from Cohen's structure theorem (note that $R$ is complete as it is artinian).
\item
Suppose that $\ch k=0$.
Then $\ch R=0$ and $R$ contains $k$ as a coefficient field.
Furthermore, for each prime ideal $\p$ of $R$, one has $\ch R_\p=\ch\kappa(\p)=0$.
Hence $R_\p$ contains $\kappa(\p)$ as a coefficient field.
\item
If $R$ is a $0$-dimensional simple singularity, $\widehat R\widehat\otimes_k\overline k$ is isomorphic to either $\overline k$ or $\overline k[\![x_0]\!]/(x_0^{n+1})$ for some $n\ge1$.
\item
The original definition \cite[Definition 9.1]{LW} of a simple singularity is different from ours.
Actually, a regular local ring is not a simple singularity in the original sense.
Even in the case of a singular local ring, by \cite[Theorems 9.2, 9.8]{LW} and Lemma \ref{23} stated below, a simple singularity in our sense is a simple singularity in the original sense, and both are the same when the local ring is complete and the residue field is algebraically closed and of characteristic zero.
We should also refer the reader to \cite[Corollary 10.19]{LW}.
\end{enumerate}
\end{rem}

Being a simple singularity is a sufficient condition for being of finite CM type.

\begin{lem}\label{23}
Every simple singularity (in our sense) has finite CM type.
\end{lem}

\begin{proof}
Let $R$ be a simple singularity.
The ring $S:=\widehat R\widehat\otimes_k\overline k$ is isomorphic to $\overline k[\![x_1,\dots,x_d]\!]$ or $\overline k[\![x_0,\dots,x_d]\!]/(f)$, where $f$ has type $T\in\{(\A_a),(\D_b),(\E_c)\,|\,a\in\Z_{\ge1},\,b\in\Z_{\ge4},\,c=6,7,8\}$.
By \cite[Theorem 9.8]{LW} and Lemma \ref{5}(1), the ring $S$ has finite CM type.
Note that the natural map $R\to S$ is a flat local homomorphism (see \cite[Theorem 49, (1)$\Leftrightarrow$(5)]{M}) and $S/\m S=\overline k$.
It follows from \cite[Theorem 10.1]{LW} that $R$ has finite CM type.
\end{proof}

Applying Theorem \ref{4} together with Lemma \ref{23}, we get a sufficient condition for the triangulated category $\d(R)$ to possess an additive generator, in terms of simple singularities.

\begin{cor}\label{19}
Let $R$ be a J-1 local ring, and suppose that $R$ is locally a simple singularity on the punctured spectrum.
Then the triangulated category $\d(R)$ possesses an additive generator.
\end{cor}

Here is an application example of the above corollary.

\begin{ex}\label{22}
Let $S=\CC[\![x,y,z]\!]$ and $R=S/(xyz)$.
As $R$ is complete, it is J-1.
For any nonmaximal prime ideal $\p$ the local ring $R_\p$ is either a regular local ring or an $(\A_1)$-singularity.
In fact, as $\p\ne\m$, we may assume $z\notin\p$, and then $R_\p\cong S_P/(xy)$, where $P$ is a preimage of $\p$ in $S$.
Hence $R$ is locally a simple singularity on the punctured spectrum, and Corollary \ref{19} guarantees the existence of objects $G$ such that $\d(R)=\add G$.
\end{ex}

\begin{rem}
Let $R$ be as in Example \ref{22}.
We can directly prove that $\d(R)=\add(R/(x)\oplus R/(y)\oplus R/(z))$, as follows.
The Jacobian ideal of $R$ is $J=(yz,xz,xy)$, so that $\sing R=\V(J)=\{\p,\q,\r,\m\}$, where $\p=(x,y)$, $\q=(x,z)$, $\r=(y,z)$ and $\m=(x,y,z)$.
Put $H=R\oplus R/(x)\oplus R/(y)\oplus R/(z)$; this is a maximal Cohen--Macaulay $R$-module.
We have $R_\p\cong A_P$, where $A=\CC[\![x,y,z]\!]/(xy)$ and $P=(x,y)$.
Let $M\in\cm(A_P)$.
Then $M=N_P$ for some $N\in\mod A$.
Since $A$ is Gorenstein, there is an exact sequence $0\to Y\to X\to N\to0$ of $A$-modules such that $X$ is maximal Cohen--Macaulay and $Y$ has finite injective dimension; see \cite[Theorem 11.17]{LW}.
Since $N_P=M$ is a maximal Cohen--Macaulay $A_P$-module, the localized exact sequence $0\to Y_P\to X_P\to N_P\to0$ splits, and $M$ is a direct summand of $X_P$.
By \cite[Proposition 14.17]{LW} (see also \cite[Proposition 2.2]{hsccm}), all the nonisomorphic indecomposable maximal Cohen--Macaulay $A$-modules are $A$, $A/(x)$, $A/(y)$, and the cokernels of the matrices $\left(\begin{smallmatrix}y&z^n\\0&x\end{smallmatrix}\right)$ and $\left(\begin{smallmatrix}x&-z^n\\0&y\end{smallmatrix}\right)$, where $n>0$.
Since $xy=0$ and $z^n$ is a unit in $A_P$, there are equivalences
$$
\left(\begin{smallmatrix}y&z^n\\0&x\end{smallmatrix}\right)
\cong\left(\begin{smallmatrix}(z^n)^{-1}y&1\\0&x\end{smallmatrix}\right)
\cong\left(\begin{smallmatrix}y&1\\0&x\end{smallmatrix}\right)
\cong\left(\begin{smallmatrix}y&1\\-xy&0\end{smallmatrix}\right)
=\left(\begin{smallmatrix}y&1\\0&0\end{smallmatrix}\right)
\cong\left(\begin{smallmatrix}0&1\\0&0\end{smallmatrix}\right)
\cong\left(\begin{smallmatrix}1&0\\0&0\end{smallmatrix}\right)
$$
of matrices over $A_P$, and similarly, $\left(\begin{smallmatrix}x&-z^n\\0&y\end{smallmatrix}\right)\cong\left(\begin{smallmatrix}1&0\\0&0\end{smallmatrix}\right)$. 
The cokernel of the matrix $\left(\begin{smallmatrix}1&0\\0&0\end{smallmatrix}\right)$ over $A_P$ is isomorphic to $A_P$.
Thus $X_P$ belongs to $\add\{A_P,(A/(x))_P,(A/(y))_P\}$.
We now see that $\cm(R_\p)=\add H_\p$.
By symmetry, we get $\cm(R_\q)=\add H_\q$ and $\cm(R_\r)=\add H_\r$.
It follows from Corollary \ref{26} that $\d(R)=\add H$ holds.
\end{rem}

From here to the end of this section, we focus on the case of a local ring of dimension one.
In this case, the punctured spectrum coincides with the set of minimal prime ideals.
Thus, as an immediate consequence of Lemmas \ref{16}, \ref{5}(1) and Theorem \ref{4}, we get the following corollary.

\begin{cor}\label{14'}
Let $R$ be a $1$-dimensional local ring which is locally a hypersurface on the punctured spectrum (e.g., a $1$-dimensional local hypersurface).
Then the triangulated category $\d(R)$ has an additive generator.
\end{cor}

For an artinian local ring $(R,\m)$ we denote by $\ell\ell(R)$ the {\em Loewy length} of $R$, i.e., $\ell\ell(R)=\min\{n\,|\,\m^n=0\}$.
Here are a couple of remarks concerning the above corollary.

\begin{rem}\label{21}
\begin{enumerate}[(1)]
\item
In the situation of Corollary \ref{14'}, we can actually take $G=\bigoplus_{\p\in\Min R,\,0<i<\ell\ell(R_\p)}R/\p^i$ as an additive generator of $\d(R)$.
Indeed, fix any $\p\in\Min R$.
The artinian local ring $R_\p$ is a hypersurface, so that it is a quotient of a discrete valuation ring (even when $R_\p$ is a field).
Using the structure theorem for finitely generated modules over a principal ideal domain, we obtain equalities $\cm(R_\p)=\mod R_\p=\add\{R_\p,R_\p/\p R_\p,\dots,R_\p/\p^{\ell\ell(R_\p)-1}R_\p\}=\add(R\oplus G)_\p$.
Corollary \ref{26} yields $\d(R)=\add(R\oplus G)=\add G$.
\item
Let $G$ be as in (1) and consider $H=R\oplus G$.
Then, for every $R$-module $M$ there exists an exact sequence $0\to K\to M\to N\to L\to0$ of $R$-modules such that $K$ has finite length, $L$ is locally free on the punctured spectrum, and $N$ is in $\add_{\mod R}H$.
In fact, write $\Min R=\{\p_1,\dots,\p_n\}$, and fix an integer $1\le i\le n$.
Put $l_i=\ell\ell(R_{\p_i})\ge1$.
As $R_{\p_i}$ is an artinian hypersurface, similarly as in (1), the structure theorem for modules over a PID gives an $R_{\p_i}$-isomorphism $M_{\p_i}\cong R_{\p_i}^{\oplus e_{i,0}}\oplus(R_{\p_i}/\p_iR_{\p_i})^{\oplus e_{i,1}}\oplus\cdots\oplus(R_{\p_i}/\p_i^{l_i-1}R_{\p_i})^{\oplus e_{i,l_i-1}}$; this is valid even when $l_i=1$.
Put $N_i=R^{\oplus e_{i,0}}\oplus(R/\p_i)^{\oplus e_{i,1}}\oplus\cdots\oplus(R/{\p_i}^{l_i-1})^{\oplus e_{i,l_i-1}}$.
We have $M_{\p_i}\cong(N_i)_{\p_i}$.
There exists an $R$-homomorphism $f_i:M\to N_i$ such that $(f_i)_{\p_i}$ is an isomorphism.
We obtain an exact sequence $0\to K\to M\xrightarrow{f}N\to L\to0$, where $K=\bigcap_{i=1}^n\ker f_i$, $N=\bigoplus_{i=1}^nN_i$ and $f$ is the transpose of $(f_1,\dots,f_n)$.
The module $N$ belongs to $\add_{\mod R}H$.
Since $(\ker f_i)_{\p_i}=0$, we have $K_{\p_i}=0$ for each $i$.
An exact sequence $0\to M_{\p_i}\xrightarrow{f_{\p_i}}N_{\p_i}\to L_{\p_i}\to0$ is induced.
Since $(0,\dots,0,(f_i)_{\p_i}^{-1},0,\dots,0)$ is a left inverse to $f_{\p_i}$, this exact sequence splits.
We see that $L_{\p_i}\cong\bigoplus_{j\ne i}(N_j)_{\p_i}=\bigoplus_{j\ne i} R_{\p_i}^{\oplus e_{j,0}}$ for each $i$.
We are done.
\end{enumerate}
\end{rem}

We relate the above corollary to the notion of finite CM$_+$ type which has been introduced in \cite{plus}.

\begin{dfn}
\begin{enumerate}[(1)]
\item
A local ring $R$ is said to have {\em finite CM$_+$ type} if it has only finitely many nonisomorphic indecomposable maximal Cohen--Macaulay modules that are not locally free on the punctured spectrum.
\item
Let $R$ be a Gorenstein local ring.
We denote by $\lcm_0(R)$ the subcategory of $\lcm(R)$ consisting of maximal Cohen--Macaulay $R$-modules $M$ such that $M_\p\cong0$ in $\lcm(R_\p)$ for all $\m\ne\p\in\spec R$.
This is a thick subcategory of the triangulated category $\lcm(R)$, so we can define the Verdier quotient $\lcm(R)/\lcm_0(R)$.
\end{enumerate}
\end{dfn}

The following corollary is immediately deduced from Corollary \ref{14'} and \cite[Proposition 4.4]{dlr}.

\begin{cor}\label{20}
Let $R$ be a Gorenstein local ring of dimension one.
Suppose that $R$ is locally a hypersurface on the punctured spectrum.
Then the triangulated category $\lcm(R)/\lcm_0(R)$ has an additive generator.
\end{cor}

\begin{rem}
If a Gorenstein local ring $R$ has finite CM$_+$ type, then it is evident that $\lcm(R)/\lcm_0(R)$ admits an additive generator.
The converse does not necessarily hold.
Indeed, let $R$ be a homomorphic image of a regular local ring.
Suppose that $R$ is a Gorenstein non-reduced ring of dimension one.
It is shown in \cite[Theorem 1.6]{plus} that $R$ has finite CM$_+$ type if and only if there are a regular local ring $S$ and a regular system of parameters $x,y$ of $R$ such that $R$ is isomorphic to either $S/(x^2)$ or $S/(x^2y)$.
Thus, by Corollary \ref{20}, the existence of an additive generator of $\lcm(R)/\lcm_0(R)$ does not necessarily imply that $R$ has finite CM$_+$ type.
\end{rem}

Here is a concrete example to explain the last part of the above remark.

\begin{ex}
Let $k$ be a field.
Consider the $1$-dimensional local hypersurface $R=k[\![x,y]\!]/(x^2y^2)$.
Then $R$ does not have finite CM$_+$ type but the category $\lcm(R)/\lcm_0(R)$ possesses an additive generator.
\end{ex}

Finally, we consider local rings of countable CM type.

\begin{dfn}
We say that a local ring $R$ has {\em countable CM type} provided that the set $\ind\cm(R)$ is at most countable.
Note by definition that finite CM type implies countable CM type.
\end{dfn}

In order to give a proof of the final result of this paper, we establish a lemma.

\begin{lem}\label{15}
\begin{enumerate}[\rm(1)]
\item
Let $R$ be a local ring with residue field $k$.
If $k$ is uncountable, so is $\kappa(\p)$ for any $\p\in\spec R$.
\item
Let $R$ be an artinian local ring whose residue field is an uncountable set.
Then $R$ has countable CM type if and only if $R$ has finite CM type, if and only if $R$ is a hypersurface.
\end{enumerate}
\end{lem}

\begin{proof}
(1) The surjection $R/\p\twoheadrightarrow R/\m=k$ shows that if $k$ is uncountable, then so is $R/\p$.
The inclusion map $R/\p\hookrightarrow\kappa(\p)$ of the domain $R/\p$ into its quotient field $\kappa(\p)$ shows that if $R/\p$ is uncountable, then so is $\kappa(\p)$.

(2) Taking Lemma \ref{5}(1) into account, it suffices to verify that $\ind\cm(R)$ is uncountable if $R$ has embedding dimension at least two.
Choose elements $x,y\in\m$ whose residue classes in $\m/\m^2$ are linearly independent over $k$.
We have $\cm(R)=\mod R$ as $R$ is artinian.
Hence any cyclic $R$-module is an indecomposable maximal Cohen--Macaulay $R$-module.
Consider the subset $\Delta=\{[R/(x+uy)]\,|\,u\in R\setminus\m\}$ of $\ind\cm(R)$.
Let $u,v$ be elements in $R\setminus\m$ such that the equality $[R/(x+uy)]=[R/(x+vy)]$ holds.
Taking the annihilators of these two cyclic $R$-modules, we obtain the equality $(x+uy)=(x+vy)$ of ideals of $R$.
We have $x+uy=a(x+vy)$ for some $a\in R$, and $x(1-a)+y(u-av)=0$ in $R$.
The choice of $x,y$ forces $\overline{1-a}=\overline{u-av}=\overline0$ in $k$.
Hence $\overline1=\overline a$ and $\overline u=\overline a\,\overline v=\overline v$ in $k$.
Therefore, the map $\pi:\Delta\to k$ given by $\pi([R/(x+uy)])=\overline u$ is well-defined.
It is evident that $\pi$ is surjective.
Since the residue field $k$ is an uncountable set, so is $\Delta$, and so is $\ind\cm(R)$.
\end{proof}

Applying Theorem \ref{4}, we obtain another sufficient condition for $\d(R)$ to have an additive generator.

\begin{cor}\label{14}
Let $R$ be a $1$-dimensional Cohen--Macaulay local ring with uncountable residue field and a canonical module.
If $R$ has countable CM type, then the triangulated category $\d(R)$ has an additive generator.
\end{cor}

\begin{proof}
Fix a nonmaximal prime ideal $\p$ of $R$.
As $R$ is a Cohen--Macaulay local ring with a canonical module, the localization $R_\p$ has countable CM type by \cite[Theorem 14.5]{LW}.
Since the residue field of $R$ is uncountable, Lemma \ref{15}(1) implies that so is $\kappa(\p)$.
Since $R_\p$ is an artinian local ring, Lemma \ref{15}(2) implies that $R_\p$ has finite CM type.
The assertion of the corollary is now a consequence of Lemma \ref{16} and Theorem \ref{4}.
\end{proof}


\end{document}